\documentclass[12pt,a4paper,reqno]{amsart}
\usepackage{hyperref}
\usepackage{url}
\usepackage{fullpage}
\usepackage{newtxtext,newtxmath}

\newcommand{\dx}{\mathrm{d}}
\newcommand{\e}{\mathrm{e}}
\newcommand{\ii}{\mathrm{i}}
\newcommand{\E}{\mathcal{E}}
\newcommand{\LL}{\mathcal{L}}
\newcommand{\M}{\mathcal{M}}
\newcommand{\R}{\mathbb{R}}
\newcommand{\Y}{\mathbb{Y}}
\newcommand{\Z}{\mathcal{Z}}
\newcommand{\Stilde}{\widetilde{S}}
\newcommand{\eqdef}{:=}

\newcommand{\Odi}[1]{\mathcal{O}\left(#1\right)}
\newcommand{\Odipm}[2]{\mathcal{O}_{#1} (#2)}

\newtheorem{Theorem}{Theorem}
\newtheorem{Lemma}{Lemma}

\author{Marco Cantarini, Alessandro Gambini and  Alessandro Zaccagnini}

\title{Ces\`aro averages for Goldbach representations \\
       with summands in arithmetic progressions}

\date{\today}


\allowdisplaybreaks

\begin{document}

\begin{abstract}
We consider weighted averages of the number of representations of an
even integer as a sum of two prime numbers, where each summand lies in
a given arithmetic progression modulo a common integer $q$.
Our result is uniform in a suitable range for $q$.
\end{abstract}

\maketitle

\section{Introduction}

In recent years a great deal of papers treated various questions
related to the average number of representations of an integer as a
sum of prime numbers, or powers of primes.
Unweighted averages have been considered by Languasco \&
Zaccagnini \cite{LanguascoZ2012a}, and more recently
Bhowmik, Halupczok, Matsumoto \& Suzuki \cite{BhowmikHMS2019}.
Weighted averages appear in a fairly large number of recent papers, as
Cantarini \cite{Cantarini2017}, \cite{Cantarini2018},
\cite{Cantarini2019}, \cite{Cantarini2017c};
Languasco \& Zaccagnini \cite{LanguascoZ2013a}, \cite{LanguascoZ2015a},
\cite{LanguascoZ2017d}, \cite{LanguascoZ2019c}, \cite{LanguascoZ2019a};
Br\"udern, Kaczorowski \& Perelli \cite{BrudernKP2019};
Goldston \& Yang \cite{GoldstonY2017}.

Here we deal with averages with a Ces\`aro weight, with the constraint
that the summands in the Goldbach representations lie in fixed
arithmetic progressions modulo a common integer $q$.
Let $q$ be a positive integer.
For integers $a$ and $b$ coprime to $q$ we define
\[
  R_G(n; q, a, b)
  \eqdef
  \sum_{\substack{m_1 + m_2 = n \\ m_1 \equiv a \bmod q \\ m_2 \equiv b \bmod q}}
    \Lambda(m_1) \Lambda(m_2),
\]
where $\Lambda$ is the usual von Mangoldt-function.
We could consider a more general definition with the summands lying in
progressions with different $q_1$ and $q_2$, but then the conditions
$m_1 \equiv a \bmod q_1$ and $m_2 \equiv b \bmod q_2$ would become
$m_1 \bmod q \in \mathcal{R}_1$ and $m_2 \bmod q \in \mathcal{R}_2$,
where $q$ is the least common multiple of $q_1$ and $q_2$ and
$\mathcal{R}_1$ and $\mathcal{R}_2$ are suitable sets of residue
classes $\bmod\, q$.

For $z = x + \ii y$ with $x \in \R^+$ and $y \in \R$ we also set
\[
  \Stilde_{a, q}(z)
  \eqdef
  \sum_{\substack{m \ge 1 \\ m \equiv a \bmod q}}
    \Lambda(m) \, \e^{-m z}.
\]
It is clear that $R_G$ is the generating function of
$\Stilde_{a, q}(z) \Stilde_{b, q}(z)$, that is
\[
  \Stilde_{a, q}(z) \Stilde_{b, q}(z)
  =
  \sum_{m \ge 1} R_G(m; q, a, b) \, \e^{-m z}.
\]
Our goal is to study averages of the quantities $R_G$, so we introduce
a real parameter $k \ge 0$ and define
\begin{equation}
\label{main-identity}
  \Sigma_k(N; q, a, b)
  \eqdef
  \sum_{n \le N}
    R_G(n; q, a, b) \frac{(N - n)^k}{\Gamma(k + 1)}
  =
  \frac1{2 \pi \ii}
  \int_{(x)}
    \e^{N z} \Stilde_{a, q}(z) \Stilde_{b, q}(z) \, \frac{\dx z}{z^{k + 1}},
\end{equation}
where $\Gamma$ is the Euler Gamma-function.
We will express $\Sigma_k$ as a sum of a main term, a secondary term
and other smaller terms, depending explicitly on the zeros of the
relevant Dirichlet $L$-functions.
The inversion of infinite integral and series in \eqref{main-identity}
can be justified as in \cite{LanguascoZ2015a} for $k > 0$.
See \S\S\ref{sec:interchange-I}--\ref{sec:interchange-II}.

This method was used in Cantarini \cite{Cantarini2017},
\cite{Cantarini2018}, \cite{Cantarini2017c}; Languasco \& Zaccagnini
\cite{LanguascoZ2013a}, \cite{LanguascoZ2015a},
\cite{LanguascoZ2017d}, \cite{LanguascoZ2019c}, \cite{LanguascoZ2019a}
to deal with weights with $k > 1$, while Br\"udern, Kaczorowski \&
Perelli \cite{BrudernKP2019} use a radically new approach in the
spirit of the classical proof of the Prime Number Theorem.

For simplicity, we denote by $\Z(\chi)$ the set of non-trivial zeros
of the Dirichlet $L$-function $L(s, \chi)$.
We can now define the terms in our decomposition:
\begin{align}
\label{def-M1}
  M^{(1)}_k(N; q)
  &=
  \frac{N^{k + 2}}{\phi(q)^2 \Gamma(k + 3)}, \\
\label{def-M2}
  M^{(2)}_k(N; q, a)
  &=
  -\frac1{\phi(q)^2}
  \sum_{\chi \bmod q}
    \overline{\chi}(a)
    \sum_{\rho \in \Z(\chi^*)}
      \frac{\Gamma(\rho) N^{k + 1 + \rho}}{\Gamma(k + 2 + \rho)}, \\
\label{def-M3}
  M^{(3)}_k(N; q, a, b)
  &=
  \frac1{\phi(q)^2}
  \sum_{\chi_1, \chi_2 \bmod q}
    \overline{\chi}_1(a) \overline{\chi}_2(b)
    \sum_{\substack{\rho_1 \in \Z(\chi_1^*) \\ \rho_2 \in \Z(\chi_2^*)}}
      \frac{\Gamma(\rho_1) \Gamma(\rho_2) N^{k + \rho_1 + \rho_2}}
           {\Gamma(k + 1 + \rho_1 + \rho_2)},
\end{align}
where here and throughout the paper $\chi^*$ denotes the primitive
character that induces $\chi$.
It is also convenient to write
\begin{equation}
\label{def-main-t}
  M_k(N; q, a, b)
  =
  M^{(1)}_k(N; q)
  +
  M^{(2)}_k(N; q, a)
  +
  M^{(2)}_k(N; q, b)
  +
  M^{(3)}_k(N; q, a, b).
\end{equation}
Throughout the paper, implicit constants may depend only on $k$.
Our main result is the following theorem.

\begin{Theorem}
\label{main-Thm}
For $k > 1$ and as $N \to +\infty$ we have
\[
  \Sigma_k(N; q, a, b)
  =
  M_k(N; q, a, b)
  +
  \Odipm{k}{(g(q)+\log N) N^{k + 1}},
\]
uniformly for $q \le (\log N)^A$ and every $a$, $b$ with
$(a, q) = (b,q) = 1$, for any fixed $A>0$ and where
\begin{equation}
\label{stima-b}
  g(q)
  =
  \begin{cases}
    \log^2(q)  & \text{if there is no exceptional zero modulo $q$,} \\
    q^{1/2} \log^2(q) & \text{if there is an exceptional zero modulo $q$.} \\
  \end{cases}
\end{equation}
\end{Theorem}

We notice that we essentially detect the second main term in Bhowmik
et al.~\cite{BhowmikHMS2019} for $k = 0$, although we prove our
Theorem~\ref{main-Thm} only for $k > 1$.
The case $q = 1$ was treated in \cite{LanguascoZ2015a}, and in fact we
closely follow its proof.
The weak uniformity bound in our Theorem and the value of
$g$ are both due to the possible exceptional (or Siegel) zero of a
Dirichlet $L$-function attached to a primitive real character modulo a
``small'' $q$.
The turning point of the proof of Theorem~\ref{main-Thm} is
Lemma~\ref{Linnik-Lemma}, which we prove in \S\ref{Lemmas}.
In fact, a weaker version would suffice in establishing the bounds
which are necessary for the exchange of single and double series with
the line integral, because in this case we may assume that $q$ is
fixed.

We remark that \eqref{main-identity} remains true also in the case
$k = 0$, with a suitable interpretation. In fact, by partial summation we have
\[
  \Stilde_{a, q}(z) \Stilde_{b, q}(z)
  =
  \sum_{n \ge 1} R_G(m; q, a, b) \, \e^{-m z}
  =
  z
  \int_0^{+\infty} f(t) \e^{-z t} \, \dx t,
\]
where
\[
  f(t)
  =
  \sum_{n \le t} R_G(n; q, a, b)
  =
  \Sigma_0(t; q, a, b).
\]
We remark that $f$ is smooth in stretches and is bounded by
an exponential.
Hence, by  standard inversion theorems (see, e.g., Theorem~8.4
of Folland \cite{Folland1992}), we have
\[
  \frac12
  \bigl( f(t+0) + f(t-0) \bigr)
  =
  \lim_{T \to +\infty}
    \frac1{2 \pi \ii}
    \int_{1/N - \ii T}^{1/N + \ii T}
      \Stilde_{a, q}(z) \Stilde_{b, q}(z) z^{-1} \e^{t z} \,\dx z.
\]

\section{Outline of the proof}

The four dominant terms in the statement, that is, in
\eqref{def-main-t}, arise multiplying formally the leading terms for
$\Stilde_{a,q}$ and $\Stilde_{b,q}$ provided by
Lemma~\ref{Linnik-Lemma}: see \eqref{def-M}.
We have to show that we may exchange the summations with integration
on the vertical line $\Re(z) = x$, and also that the error term is
small.
We need the identity
\begin{equation}
\label{Gamma-transf}
  \frac1{2 \pi \ii}
  \int_{(a)} u^{-s} \e^u \, \dx u
  =
  \frac1{\Gamma(s)}
\end{equation}
for $a > 0$ and its variant
\[
  \frac1{2 \pi}
  \int_{\R} \frac{\e^{\ii D u}}{(a + \ii u)^s} \, \dx u
  =
  \begin{cases}
    D^{s - 1} \e^{-a D} / \Gamma(s) & \text{if $D > 0$,} \\
    0 & \text{if $D < 0$,} \\
  \end{cases}
\]
in order to prove that \eqref{main-identity} holds.
Here $\Re(a) > 0$ and $\Re(s) > 0$.
For the proof, see de Azevedo Pribitkin \cite{Azevedo2002}.
For brevity, we write
\begin{equation}
\label{def-M-E}
  \Stilde_{a,q}(z)
  =
  \frac1{\phi(q)}
  \bigl( \M(z; q, a) + \E(z; q, a) \bigr),
\end{equation}
where
\begin{equation}
\label{def-M}
  \M(z; q, a)
  =
  \frac1z
  -
  \sum_{\chi \bmod q}
    \overline{\chi}(a)
    \sum_{\rho \in \Z(\chi^*)} z^{-\rho} \Gamma(\rho).
\end{equation}
The bound for $\E(z; q, a)$ is provided by Lemma~\ref{Linnik-Lemma}.
We now substitute into \eqref{main-identity} and find that
\begin{align}
\label{identity-Sigma-1}
  \Sigma_k(N; q, a, b)
  =
  \frac1{2 \pi \ii \phi^2(q)}
  \Bigl(
  &\int_{(x)}
    \e^{N z} \M(z; q, a) \M(z; q, b) \, \frac{\dx z}{z^{k + 1}} \\
\label{identity-Sigma-2}
  \qquad+
  &\int_{(x)}
    \e^{N z}
    \bigl( \M(z; q, a) \E(z; q, b) + \E(z; q, a) \M(z; q, b) \bigl) \,
    \frac{\dx z}{z^{k + 1}} \\
\label{identity-Sigma-3}
  \qquad+
  &\int_{(x)}
    \e^{N z} \E(z; q, a) \E(z; q, b) \, \frac{\dx z}{z^{k + 1}}
  \Bigr).
\end{align}
Expanding the first term in \eqref{identity-Sigma-1}, exchanging
summation with integration and using identity \eqref{Gamma-transf}, we
recover the terms in \eqref{def-M1}, \eqref{def-M2} and
\eqref{def-M3}, that is, the main term $M_k(N; q, a, b)$ defined
in~\eqref{def-main-t}.
The proofs that the exchanges are legitimate are in
\S\S\ref{sec:interchange-I}--\ref{sec:interchange-II},
while the proofs that single and double sums over zeros in
\eqref{def-M2} and \eqref{def-M3} respectively converge
follow closely the argument in \cite{LanguascoZ2015a} and
\cite{LanguascoZ2019c}, with minor modifications in the choices of the
regions.
We do not include details.

Now we deal with the error terms.
We remark that
\begin{equation}
\label{z-bound}
  \vert z \vert^{-1}
  \ll
  \begin{cases}
    x^{-1}             & \text{if $\vert y \vert \le x$,} \\
    \vert y \vert^{-1} & \text{if $\vert y \vert >   x$.}
  \end{cases}
\end{equation}
We will eventually choose $x = N^{-1}$.
Hence, by Lemma~\ref{Linnik-Lemma} or simply by the Brun-Titchmarsh
inequality we have
\[
  \Stilde_{a,q}(z)
  \ll
  \frac1{\phi(q) x},
\]
provided that $q \le (\log N)^A$ for any fixed $A > 0$.
Hence
\[
  \vert \M(z; q, a) \vert
  \le
  \phi(q) \Stilde(x; q, a)
  +
  \vert \E(z; q, a) \vert
  \ll
  x^{-1}
  +
  \vert \E(z; q, a) \vert.
\]
Therefore we have
\begin{align*}
  \Bigl\vert
  \int_{(x)}
    \e^{N z} \M(z; q, a) \E(z; q, b) \, \frac{\dx z}{z^{k + 1}}
  \Bigr\vert
  &\ll
  \e^{N x}
  \int_{(x)}
    \vert \M(z; q, a) \E(z; q, b) \vert \,
    \frac{\dx y}{\vert z \vert^{k + 1}} \\
  &\ll
  \e^{N x}
  \int_{(x)}
    \bigl( x^{-1} + \vert \E(z; q, a) \vert \bigr) \,
    \vert \E(z; q, b) \vert \,
    \frac{\dx y}{\vert z \vert^{k + 1}}.
\end{align*}
For brevity we set
\begin{equation}
\label{L-def}
  \LL(x; q, a)
  :=
  \int_{(x)}
    \vert \E(z; q, a) \vert^2 \,
    \frac{\dx y}{\vert z \vert^{k + 1}}.
\end{equation}
By the Cauchy-Schwarz inequality, the total contribution of the terms
in \eqref{identity-Sigma-2} and~\eqref{identity-Sigma-3} is
\begin{align*}
  &\ll
  \e^{N x} x^{-1}
  \int_{(x)}
    \bigl( \vert \E(z; q, a) \vert + \vert \E(z; q, b) \vert \bigr) \,
    \frac{\dx y}{\vert z \vert^{k + 1}}
  +
  \e^{N x}
  \int_{(x)}
    \vert \E(z; q, a) \vert \cdot \vert \E(z; q, b) \vert \,
    \frac{\dx y}{\vert z \vert^{k + 1}} \\
  &\ll
  \e^{N x} x^{-1}
  \sum_{h \in \{a, b \}}
    \Bigl(
      \int_{(x)} \frac{\dx y}{\vert z \vert^{k + 1}}
      \LL(x; q, h)
    \Bigr)^{1/2}
  +
  \e^{N x}
  \Bigl( \LL(x; q, a) \LL(x; q, b)
  \Bigr)^{1/2} \\
  &\ll
  \e^{N x}
  \max_{(h, q) = 1}
  \Bigl(
  x^{- 1 - k / 2}
  \LL(x; q, h)^{1 / 2}
  +
  \LL(x; q, h)
  \Bigr),
\end{align*}
by \eqref{z-bound} again.
We choose $x = 1 / N$ and use Lemma~\ref{L-bound} to bound \eqref{L-def}.
We will complete the proof of Theorem~\ref{main-Thm} in \S\ref{proof-thm-2}.

\section{Lemmas}
\label{Lemmas}

For a Dirichlet character $\chi \bmod q$ let
\[
  \Stilde(z; \chi)
  \eqdef
  \sum_{m \ge 1} \chi(m) \Lambda(m) \e^{-m z},
\]
so that, by orthogonality,
\begin{equation}
\label{Stilde-Schi}
  \Stilde_{a, q}(z)
  =
  \frac1{\phi(q)}
  \sum_{\chi \bmod q}
    \overline{\chi}(a) \Stilde(z; \chi).
\end{equation}
We now express $\Stilde(\chi)$ by means of $\Stilde(\chi^*)$, where
$\chi^* \bmod q^*$ is the primitive character that induces $\chi$.
We have
\begin{equation}
\label{trans-to-chi*}
  \Bigl\vert
    \Stilde(z; \chi) - \Stilde(z; \chi^*)
  \Bigr\vert
  \le
  \sum_{\substack{m \ge 1 \\ (m, q) > 1}} \Lambda(m) \e^{-m x}
  =
  \sum_{p \mid q} \log(p)
    \sum_{\nu \ge 1} \e^{-p^\nu x}
  \le
  \e^{- x} \log(q).
\end{equation}

We recall some properties of the Dirichlet $L$-functions that we need
in the proof of Lemma~\ref{Linnik-Lemma}.
First, let $\chi$ be a primitive odd character modulo $q > 1$ and let
\[
  b(\chi)
  =
  \frac{L'}L (0, \chi)
  =
  -
  \frac12 \log \frac q{\pi}
  -
  \frac12 \frac{\Gamma\,'}{\Gamma}\Bigl( \frac12 \Bigr)
  +
  B(\chi),
\]
where $b(\chi)$ is defined in \S19 of Davenport \cite{Davenport2000} and $B(\chi)$ appears in the Weierstrass product for $L(s, \chi)$.
If $\chi$ is even, we let
\[
  b(\chi)
  =
  \lim_{s \to 0}
    \Bigl( \frac{L'}L (s, \chi) - \frac1s \Bigr)
  =
  -
  \frac12 \log \frac q{\pi}
  +
  B(\chi)
  -
  \lim_{s \to 0}
    \Bigl( \frac12 \frac{\Gamma\,'}{\Gamma}\Bigl( \frac s2 \Bigr) + \frac1s
    \Bigr).
\]

By the argument on pages 118--119 of Davenport \cite{Davenport2000} we
have
\[
  b(\chi)
  =
  \Odi{\log(q)}
  -
  \sum_{\vert \gamma \vert < 1} \frac1\rho.
\]
The Riemann-von Mangoldt formula for the number of zeros implies that
the number of summands on the right is $\ll \log(q)$.
Each of the summands is $\ll \log(q)$, unless $\chi$ is a real
character such that the relative $L$-function has an exceptional zero
$\widetilde \beta \in [1 - c / \log(q), 1]$.
Hence
\[
  b(\chi)
  =
  \Odi{\log^2(q)}
  -
  \frac1{\widetilde \beta}
  -
  \frac1{1 - \widetilde \beta}
  =
  \Odi{\log^2(q)}
  -
  \frac1{1 - \widetilde \beta}.
\]
The terms containing the exceptional zero are to be omitted if it
does not exist.
We now recall the upper bound
$(1 - \widetilde \beta)^{-1} \ll q^{1 / 2} \log^2(q)$
which is (12) on page 96 of Davenport.
Hence we have $b(\chi) = \Odi{g(q)}$.
Furthermore, for $\sigma = \Re(w) \in [-1, 2]$, by (4) in \S16 of
\cite{Davenport2000} we have
\[
  \frac{L'}{L}(w, \chi)
  =
  \sum_{\substack{\rho \\ \vert t - \gamma \vert < 1}}
    \frac1{w - \rho}
  +
  \Odi{\log(q (\vert t \vert + 2))}.
\]
For $w$ on the line $\Re(s) = -\frac12$ the summands are uniformly
bounded, and their number is $\ll \log(q (\vert t \vert + 2))$
by the Riemann-von Mangoldt formula for the $L$-functions.
Hence
\begin{equation}
\label{bound-L'/L}
  \frac{L'}L \Bigl( -\frac12 + \ii t, \chi \Bigr)
  \ll
  \log\bigl( q (\vert t \vert + 2) \bigr).
\end{equation}

\begin{Lemma}
\label{Linnik-Lemma}
Let $\M(z; q, a)$ and $\E(z; q, a)$ be defined by \eqref{def-M-E}
and \eqref{def-M}.
Then
\begin{equation}
\label{E-bound-q}
  \E(z; q, a)
  \ll
  g(q) + 1
  +
  \log(\vert z \vert)
  +
  \vert z \vert^{1 / 2} \bigl( \log(q) + 1 \bigr)
  \cdot
  \begin{cases}
    1
    &\text{if $\vert y \vert \le x$,} \\
    1 + \log^2\bigl( \vert y \vert / x \bigr)
    &\text{if $\vert y \vert > x$.}
  \end{cases}
\end{equation}
The implicit constant is absolute.
\end{Lemma}

\begin{proof}
Assume for the time being that $\chi$ is a primitive character, and
let $\delta(\chi) = 1$ if $\chi$ is principal and $\delta(\chi) = 0$
otherwise.
Following the proof in \S4 of Linnik \cite{Linnik1946}, or Lemma~4.1
of \cite{LanguascoZ2015a}, we have
\[
  \Stilde(z; \chi)
  =
  -
  \frac1{2 \pi \ii}
  \int_{2-\ii\infty}^{2+\ii\infty}
    z^{-w} \Gamma(w) \frac{L'}L (w, \chi) \, \dx w
  =
  \frac{\delta(\chi)}z
  -
  \sum_{\rho \in \Z(\chi)} z^{-\rho} \Gamma(\rho)
  +
  Q(z; \chi)
  +
  R(z; \chi),
\]
say, where
\begin{equation*}
  R(z; \chi)
  \eqdef
  -
  \frac1{2 \pi \ii}
  \int_{- 1 / 2 - \ii \infty}^{- 1 / 2 + \ii \infty}
    z^{-w} \Gamma(w) \frac{L'}L (w, \chi) \, \dx w
\end{equation*}
and
\begin{equation}
\label{def-Q}
  Q(z; \chi)
  \eqdef
  \begin{cases}
    - (L'/L) (0, \chi)          &\text{if $\chi$ is an odd character} \\
    -\gamma + b(\chi) + \log(z) &\text{if $\chi$ is an even character,}
  \end{cases}
\end{equation}
taking into account the double pole of the integrand at $s = 0$.
Hence, recalling \eqref{Stilde-Schi} and \eqref{trans-to-chi*}, we have
\begin{align*}
  \Stilde_{a, q}(z)
  &=
  \frac1{\phi(q)}
  \sum_{\chi \bmod q}
    \overline{\chi}(a) \Stilde(z; \chi^*)
  +
  \Odi{\frac1{\phi(q)} \sum_{\chi \bmod q}
    \Bigl\vert
      \Stilde(z; \chi) - \Stilde(z; \chi^*)
    \Bigr\vert} \\
  &=
  \frac1{\phi(q) z}
  -
  \frac1{\phi(q)}
  \sum_{\chi \bmod q}
    \overline{\chi}(a)
    \sum_{\rho \in \Z(\chi^*)} z^{-\rho} \Gamma(\rho) \\
\notag
  &\qquad\qquad+
  \Odi{
    \frac1{\phi(q)}
    \sum_{\chi \bmod q}
      \bigl( \vert R(z; \chi^*) \vert + \vert Q(z; \chi^*) \vert \bigr)
    +
    \e^{-x} \log(q)} \\
  &=
  \frac1{\phi(q)} \M(z; q, a)
  +
  \Odi{
    \frac1{\phi(q)}
    \sum_{\chi \bmod q}
      \bigl( \vert R(z; \chi^*) \vert + \vert Q(z; \chi^*) \vert \bigr)
    +
    \e^{-x} \log(q)}.
\end{align*}
In order to treat $R(z; \chi)$ we need the bound \eqref{bound-L'/L}
and $\Gamma(w) \ll \vert t \vert^{-1} \e^{-\pi \vert t \vert / 2}$,
(see Titchmarsh \cite{Titchmarsh1988} \S4.42)
which is valid for $\vert t \vert \to +\infty$.
We split the line $\sigma = -1/2$ into the set
$L_c = \{ w = -1/2 + \ii t \colon \vert t \vert > c \}$, where $c > 0$
is a suitable large absolute constant, and its complement.
Arguing as in the proof of Lemma~4.1 of \cite{LanguascoZ2015a}, and
examining various cases we find
\[
  \int_{L_c}
    z^{-w} \Gamma(w) \frac{L'}L (w, \chi) \, \dx w
  \ll
  \vert z \vert^{1 / 2}
  \cdot
  \begin{cases}
    \log (q) + 1 &\text{if $\vert y \vert \le x$,} \\
      \log^2( \vert y \vert / x ) + 1
      +
      \log(q) ( \log( \vert y \vert / x) + 1)
    &\text{if $\vert y \vert > x$.}
  \end{cases}
\]
Finally, the integration over the complement of $L_c$ yields
a contribution $\Odi{1 + \log(q)}$.

We now turn to the estimation of $\vert Q(z; \chi) \vert$.
If $\chi$ is an odd character, then we have
\[
  \frac{L'}L (0, \chi)
  \ll
  \log(q),
\]
by formula 10.35 of \S10 of Montgomery \& Vaughan
\cite{MontgomeryV2007}.
If $\chi$ is even, we argue as above in \eqref{stima-b}, and by
\eqref{def-Q} we have
\[
  Q(z; \chi)
  \ll
  g(q)
  +
  \vert \log(z) \vert.
\]
This implies that $\E$ satisfies the bound in \eqref{E-bound-q}.
\end{proof}

\begin{Lemma}
\label{L-bound}
For $k > 1$ we have
\[
  \LL(N^{-1}; q, a)
  :=
  \int_{(1 / N)}
    \vert \E(z; q, a) \vert^2 \,
    \frac{\dx y}{\vert z \vert^{k + 1}}
  \ll_{k}
  N^{k}\left(g\left(q\right)+\log\left(N\right)\right)^{2}.
\]
\end{Lemma}

\begin{proof}
We have
\begin{align*}
  \LL(N^{-1}; q, a)
  &\ll
  N^{k+1}\int_{-1/N}^{1/N}\left|\E\left(z;q,a\right)\right|^{2}\dx y
  +
  \int_{1/N}^{+\infty}
    \frac{\left|\E\left(z;q,a\right)\right|^{2}}{y^{k+1}}\dx y \\
  &:=
  N^{k+1}I_{1}+I_{2},
\end{align*}
say.
By Lemma \ref{Linnik-Lemma} we easily get
\begin{align*}
  I_{1}
  &\ll
  \int_{-1/N}^{1/N}
    \left(g\left(q\right)+\log\left(N\right)+N^{-1/2}\left(\log\left(q\right)+1\right)\right)^{2}\dx y \\
  &\ll
  N^{-1}\left(g\left(q\right)+\log\left(N\right)\right)^{2}.
\end{align*}
Now we analyze $I_{2}$. Using again Lemma \ref{Linnik-Lemma} we obtain
\begin{align*}
  I_{2}
  &\ll
  \int_{1/N}^{+\infty}\frac{\left(g\left(q\right)+1+\left|\log\left(y\right)\right|+y^{1/2}\left(\log\left(q\right)+1\right)\left(1+\log^{2}\left(Ny\right)\right)\right)^{2}}{y^{k+1}}\dx y\\
  &\ll
  \int_{1/N}^{+\infty}\frac{\left(g\left(q\right)+1+\left|\log\left(y\right)\right|\right)^{2}}{y^{k+1}}\dx y+\left(\log\left(q\right)+1\right)^{2}\int_{1/N}^{+\infty}\frac{\left(1+\log^{2}\left(Ny\right)\right)^{2}}{y^{k}}\dx y.
\end{align*}
The first summand can be easily estimated:
\begin{align*}
  \int_{1/N}^{+\infty}\frac{\left(g\left(q\right)+1+\left|\log\left(y\right)\right|\right)^{2}}{y^{k+1}}\dx y
  &\ll_{k}
  N^{k}\left(g\left(q\right)+1\right)^{2}
  +
  \int_{1/N}^{+\infty}\frac{\log^2\left(y\right)}{y^{k+1}}\dx y\\
  &\ll_{k}
  N^{k}\left(\left(g\left(q\right)+1\right)^{2}+\log^{2}\left(N\right)\right).
\end{align*}

We use the change of variables $v = N y$ for the integral of the last
summand and we get
\begin{align*}
  \int_{1/N}^{+\infty}\frac{\left(1+\log^{2}\left(Ny\right)\right)^{2}}{y^{k}}\dx y
  &\ll_{k}
  \left(N^{k-1}+\int_{1/N}^{+\infty}\frac{\log^{4}\left(Ny\right)}{y^{k}}\dx y\right)\\
  &\ll_{k}
  N^{k-1}\left(1+\int_{1}^{+\infty}\frac{\log^{4}\left(v\right)}{v^{k}}\dx v\right)\ll_{k}
  N^{k-1}
\end{align*}
since $k>1$, and Lemma~\ref{L-bound} follows.
\end{proof}

\section{Completion of the proof of Theorem~\ref{main-Thm}}
\label{proof-thm-2}

In order to complete the proof it is enough to see that by Lemma
\ref{L-bound} we have
\begin{align*}
  \max_{(h, q) = 1}&
  \Bigl(
  N^{ 1 + k / 2}
  \LL(N^{-1}; q, h)^{1 / 2}
  +
  \LL(N^{-1}; q, h)
  \Bigr)
  \\
  &\ll
  N^{k+1}(g(q)+\log (N))+N^k(g(q)+\log (N))^2
  \\
 & \ll
  N^{k+1}(g(q)+\log (N)),
\end{align*}
where the last estimation follows from the bound for $q$.

\section{Interchange of the series over zeros with the line integral}
\label{sec:interchange-I}

It remains to prove that all the interchanges of the series and the
integrals are legitimate.

We now state without proof two technical lemmas that can be proved by
means of small variations on the arguments in Lemmas~2 and 3 in
\cite{LanguascoZ2015a}.
In both cases, we assume that $\chi \bmod q$ is a primitive character
and we let $\rho_\chi = \beta_\chi + \ii \gamma_\chi$ run over the
non-trivial zeros of the associated Dirichlet $L$-function.
Furthermore, we let $\alpha > 1$ be a parameter.

\begin{Lemma}
\label{Lemma-scambio-1}
We have
\[
  \sum_{\gamma_\chi > 0}
    \gamma_\chi^{\beta_\chi-1/2}
    \int_1^{+\infty}
      \log^c(u)
      \exp\Bigl( -\gamma_\chi \arctan \Bigl( \frac1u \Bigr) \Bigr)
      \, \frac {\dx u}{u^{\alpha + \beta_\chi}}
  \ll_{\alpha, c}
  \sum_{\gamma_\chi > 0}
    \gamma_\chi^{1/2 - \alpha}
\]
for all $c \ge 0$.
The series on the left converges for $\alpha > 3 / 2$ and diverges
otherwise.
\end{Lemma}

\begin{Lemma}
\label{Lemma-scambio-2}
Let $z = x + \ii y$ with $x \in (0, 1)$ and $y \in \R$.
We have
\begin{multline*}
  \sum_{\rho_\chi}
    \vert \gamma_\chi \vert^{\beta_\chi - 1/2}
    \int_{\Y}
      \log^c \Bigl( \frac{\vert y \vert}x \Bigr)
      \exp\Bigl( \gamma_\chi \arctan \Bigl( \frac yx \Bigr)
                 - \frac\pi2 \vert \gamma_\chi \vert
          \Bigr)
      \, \frac {\dx y}{\vert z \vert^{\alpha + \beta_\chi}} \\
  \ll_{\alpha, c}
  x^{-\alpha}
  \sum_{\gamma_\chi > 0}
    \gamma_\chi^{\beta_\chi - 1/2}
    \exp \Bigl( -\frac \pi4 \gamma_\chi \Bigr),
\end{multline*}
for all $c \ge 0$, where $\Y = \Y_1 \cup \Y_2$,
$\Y_1 = \{ y \in \R \colon y \gamma_\chi \le 0 \}$
and
$\Y_2 = \{ y \in [-x, x] \colon y \gamma_\chi > 0 \}$.
\end{Lemma}

Let $z = x + \ii y$ and $w = u + \ii v$, where $u \ge 0$.
We recall that
\begin{equation}
\label{complex-pow}
  \vert z^w \vert
  =
  \vert z \vert^u \exp \Bigl( v \arctan \Bigl( \frac yx \Bigr) \Bigr)
\end{equation}
and
\begin{equation}
\label{Stirling}
  \vert \Gamma(w) \vert
  \le
  (2 \pi)^{1 / 2}
  \vert w \vert^{u - 1 / 2}
  \e^{-\pi \vert v \vert / 2}
  \exp \Bigl( \frac1{6 \vert w \vert} \Bigr).
\end{equation}
The latter form of the Stirling formula can be found in Olver et
al.~\cite{OlverLBC2010}, Lemma 5.6(ii) equation 5.6.9.
We want to establish  the absolute convergence of
\[
  \sum_{\chi \bmod q}
    \sum_{\rho \in \Z(\chi^*)}
      \vert \Gamma(\rho) \vert
      \int_{(1/N)}
        \vert \e^{N z} \vert \cdot
        \vert z^{- k - 2 - \rho} \vert \cdot
        \vert \dx z \vert.
\]
Using \eqref{complex-pow} and \eqref{Stirling} we find
\[
  \vert \Gamma(\rho) \vert
  \int_{(1/N)}
    \vert z^{- k - 2 - \rho} \vert \cdot
    \vert \dx z \vert
  \ll
  \vert \gamma \vert^{\beta - 1 / 2}
  \exp \Bigl( \frac1{6 \vert \gamma \vert} \Bigr)
  \int_{\R}
    \exp\Bigl( \gamma \arctan(N y) - \frac\pi2 \vert \gamma \vert \Bigr)
    \, \frac{\dx y}{\vert z \vert^{k + 2 + \beta}}.
\]
We sum over $\rho \in \Z(\chi^*)$ and we can assume by symmetry that $\gamma>0$.  Recalling Lemma~\ref{Lemma-scambio-2}, we have
\[
  \sum_{\gamma>0}
    \vert \Gamma(\rho) \vert
    \int_{\Y}
       \vert \e^{N z} \vert \cdot
       \vert z^{- k - 2 - \rho} \vert \cdot
       \vert \dx z \vert
  \ll_k
  N^{k + 2}
  \sum_{\gamma > 0}
    \gamma^{\beta - 1 / 2} \exp \Bigl( -\frac\pi4 \gamma\Bigr).
\]
We still have to deal with the case $\gamma > 0$ and $y > 1 / N$.
Using the identity $\arctan(y) + \arctan(1/y) = \pi/2$, we have
\begin{align*}
  \sum_{\gamma>0}
    \vert \Gamma(\rho) \vert
    &\int_{1 / N}^{+\infty}
       \vert \e^{N z} \vert \cdot
       \vert z^{- k - 2 - \rho} \vert \cdot
       \vert \dx z \vert \\
  &\ll
  \sum_{\gamma > 0}
    \gamma^{\beta - 1 / 2}
    \int_{1 / N}^{+\infty}
      \exp\Bigl( \gamma \arctan(N y) - \frac\pi2 \gamma \Bigr)
      \, \frac{\dx y}{y^{k + 2 + \beta}} \\
  &\ll
  \sum_{\gamma > 0}
    \gamma^{\beta - 1 / 2}
    \int_{1 / N}^{+\infty}
      \exp\Bigl( - \gamma \arctan(1 / (N y)) \Bigr)
      \, \frac{\dx y}{y^{k + 2 + \beta}} \\
  &=
  \sum_{\gamma > 0}
    \gamma^{\beta - 1 / 2}
    N^{k + 1 + \beta}
    \int_1^{+\infty}
      \exp\Bigl( - \gamma \arctan(1 / u) \Bigr)
      \, \frac{\dx u}{u^{k + 2 + \beta}}.
\end{align*}
Using  Lemma~\ref{Lemma-scambio-1} we see that this is
\[
  \ll_k
  N^{k + 2}
  \sum_{\gamma > 0} \gamma^{-k - 3 / 2},
\]
which converges for $k > - 1/2$, by standard zero-density estimates.

\section{Interchange of the double series over zeros with the line integral}
\label{sec:interchange-II}

We start examining
\[
  \sum_{\chi_1 \bmod q}
  \sum_{\rho_1 \in \Z(\chi_1^*)}
    \vert \Gamma(\rho_1) \vert
    \int_{(1/N)}
      \vert \e^{N z} \vert \cdot
      \vert z^{- k - 1 - \rho_1} \vert \cdot
      \Bigl\vert
        \sum_{\chi_2 \bmod q}
          \overline{\chi}_2(b)
          \sum_{\rho_2 \in \Z(\chi_2^*)}
            \Gamma(\rho_2) z^{-\rho_2}
      \Bigr\vert
    \cdot \vert \dx z \vert.
\]
Lemma~\ref{Linnik-Lemma} implies that
\begin{gather*}
  \Bigl\vert
    \sum_{\chi_2 \bmod q}
    \overline{\chi}_2(b)
    \sum_{\rho_2 \in \Z(\chi_2^*)}
      \Gamma(\rho_2) z^{-\rho_2}
  \Bigr\vert
  \le
  \phi(q)
  \vert \Stilde_{a,q}(z) \vert
  +
  \frac1{\vert z \vert}
  +
  \sum_{\chi_2 \bmod q}
    \bigl( \vert R(z; \chi_2^*) \vert + \vert Q(z; \chi_2^*) \vert \bigr) \\
\ll_{q}\frac{1}{\left|z\right|}+N
+
\log\left(\left|z\right|\right)+\left|z\right|^{1/2}
\cdot
  \begin{cases}
    1 & \left|y\right|\leq1/N\\
    1+\log^{2}\left(N\left|y\right|\right) & \left|y\right|>1/N,
  \end{cases}
\end{gather*}
hence we have to study
\[
  \sum_{\chi_1 \bmod q}
  \sum_{\rho_1 \in \Z(\chi_1^*)}
    \vert \Gamma(\rho_1) \vert(A_{1}+A_{2}+A_{3})
\]
where
\begin{align*}
A_1=&
    \int_{-1/N}^{1/N}
      N^{ k + 1 + \beta_1}  \exp(\gamma_{1} \arctan\left(Ny\right))
      \bigg( N
  +
\log(N) +
  N^{-1 / 2} \bigg)
   \dx y,
\\
A_2=&
    \int_{1/N}^{+\infty}
      y^{- k - 1 - \beta_1} \exp(\gamma_{1} \arctan\left(Ny\right))
      \bigg( \frac1{y}+ N
  +
\vert \log(y) \vert+
  y^{1 / 2}(1+\log^{2}(Ny)) \bigg)
   \dx y,
   \\
A_3=  &
    \int_{-\infty}^{-1/N}
      \hspace{-.2cm}\vert y \vert^{- k - 1 - \beta_1} \exp(\gamma_{1} \arctan\left(Ny\right))
      \bigg( \frac1{\vert y \vert}+ N
  +
\vert \log(\vert y \vert) \vert+
 \vert y \vert^{1 / 2}(1+\log^{2}(N\vert y \vert)) \bigg)
   \dx y.
\end{align*}

By symmetry, we can assume that $\gamma_{1}>0$. Clearly, by \eqref{Stirling}, we have
\[
  \sum_{\chi_1 \bmod q}
  \sum_{\rho_1 \in \Z(\chi_1^*)}
    \vert \Gamma(\rho_1) \vert A_{1}\ll_{q} N^{k+2} \sum_{\chi_1 \bmod q}\sum_{\gamma_{1}>0}  \gamma_{1}^{\beta_{1}-1/2}\exp{\bigg(\frac{-\pi\, \gamma_{1} }{4}\bigg)}
\]
and the series is obviously convergent, by standard density estimates.

Now, we consider $A_{2}$: we have, again by \eqref{Stirling}, that
\begin{align*}
  &\sum_{\chi_1 \bmod q}
  \sum_{\rho_1 \in \Z(\chi_1^*)}
    \vert \Gamma(\rho_1) \vert
A_{2}\ll_{q} \sum_{\chi_{1}\,\mod\,q}\sum_{\gamma_{1}>0}\gamma_{1}^{\beta_{1}-1/2}
\\
&\times\int_{1/N}^{+\infty}y^{-k-1-\beta_{1}}\exp\left(\gamma_{1}\left(\arctan\left(Ny\right)-\frac{\pi}{2}\right)\right)
\left(\frac{1}{y}+N+\left|\log\left(y\right)\right|+y^{1/2}\left(1+\log^{2}\left(Ny\right)\right)\right) \dx y.
\end{align*}
Then, taking $v=Ny$, we obtain
\begin{align*}
  &\sum_{\chi_1 \bmod q}
  \sum_{\rho_1 \in \Z(\chi_1^*)}
    \vert \Gamma(\rho_1) \vert
A_{2}\ll_{q} N^{k}\sum_{\chi_{1}\,\mod\,q}\sum_{\gamma_{1}>0}N^{\beta_{1}}\gamma_{1}^{\beta_{1}-1/2}
\\
&\times\int_{1}^{+\infty}v^{-k-1-\beta_{1}}\exp\left(-\gamma_{1}\arctan\left(\frac{1}{v}\right)\right)  \left(\frac{N}{v}+N+\log\left(vN\right)+v^{1/2}N^{-1/2}\left(1+\log^{2}\left(v\right)\right)\right)\dx v
\end{align*}
and now, by Lemma \ref{Lemma-scambio-1}, it is easy to see that the integral and the series are convergent for $k>1$. If $y\in\left(-\infty,-1/N\right)$ the result is trivial since $\exp\left(\gamma_{1}\left(\arctan\left(Ny\right)-\frac{\pi}{2}\right)\right)\leq\exp\left(-\gamma_{1}\frac{\pi}{2}\right)$ and the condition $k>1$ suffices to ensure the convergence of the integral.

Finally, if we consider
\[
\sum_{\chi_{1}\,\mod\,q}\sum_{\chi_{2}\,\mod\,q}\sum_{\rho_{1}\in Z\left(\chi_{1}^{\star}\right)}\left|\Gamma\left(\rho_{1}\right)\right|\sum_{\rho_{2}\in Z\left(\chi_{2}^{\star}\right)}\left|\Gamma\left(\rho_{2}\right)\right|\int_{\left(1/N\right)}\left|e^{Nz}\right|\left|z^{-k-1-\rho_{1}-\rho_{2}}\right|\left|\dx z \right|,
\]
we can simply argue as in \cite{LanguascoZ2015a}, equation $(21)$,
getting the convergence for $k > 1$.

%\bibliography{Bibliografia}
%\bibliographystyle{amsplain}

\providecommand{\bysame}{\leavevmode\hbox to3em{\hrulefill}\thinspace}
\providecommand{\MR}{\relax\ifhmode\unskip\space\fi MR }
\providecommand{\MRhref}[2]{%
  \href{http://www.ams.org/mathscinet-getitem?mr=#1}{#2}
}
\providecommand{\href}[2]{#2}

\bigskip

\begin{tabular}{l}
Marco Cantarini \\
Dipartimento di Matematica e Informatica \\
Universit\`a di Perugia \\
Via Vanvitelli, 1 \\
06123, Perugia, Italia \\
email (MC): \texttt{marco.cantarini@unipg.it} \\[10pt]
Alessandro Gambini\\
Dipartimento di Matematica Guido Castelnuovo \\
Sapienza Universit\`a di Roma \\
Piazzale Aldo Moro, 5 \\
00185 Roma, Italia \\
email (AG): \texttt{alessandro.gambini@uniroma1.it} \\[10pt]
Alessandro Zaccagnini \\
Dipartimento di Scienze, Matematiche, Fisiche e Informatiche \\
Universit\`a di Parma \\
Parco Area delle Scienze, 53/a \\
43124 Parma, Italia \\
email (AZ): \texttt{alessandro.zaccagnini@unipr.it}
\end{tabular}

\end{document}